\documentclass[12pt]{amsart}
\usepackage{amsfonts}
\usepackage{amsmath}
\usepackage{amssymb}
\usepackage{graphicx}%
\providecommand{\U}[1]{\protect\rule{.1in}{.1in}}
\newtheorem{theorem}{Theorem}

\newtheorem{corollary}[theorem]{Corollary}

\newtheorem{definition}[theorem]{Definition}

\newtheorem{lemma}[theorem]{Lemma}
\newtheorem{notation}[theorem]{Notation}

\newtheorem{proposition}[theorem]{Proposition}
\newtheorem{remark}[theorem]{Remark}

\begin{document}

\title{On the Positivity of Trace Class Operators}
\author{Elena Cordero}
\address{Dipartimento di Matematica, Universit\`a di Torino, Dipartimento di
Matematica, via Carlo Alberto 10, 10123 Torino, Italy}
\email{elena.cordero@unito.it}
\thanks{}
\author{Maurice de Gosson}
\address{University of Vienna, Faculty of Mathematics,
Oskar-Morgenstern-Platz 1 A-1090 Wien, Austria}
\email{maurice.de.gosson@univie.ac.at}
\thanks{}
\author{Fabio Nicola}
\address{Dipartimento di Scienze Matematiche, Politecnico di Torino, corso
Duca degli Abruzzi 24, 10129 Torino, Italy}
\email{fabio.nicola@polito.it}
\subjclass[2010]{46E35, 35S05, 81S30, 42C15}
\keywords{Wigner transform, trace class operator, positive operator, Weyl
symbol, Gabor frames}

\maketitle

\begin{abstract}
The characterization of positivity properties of Weyl operators is a
notoriously difficult problem, and not much progress has been made since the
pioneering work of Kastler, Loupias, and Miracle-Sole (KLM). In this paper we
begin by reviewing and giving simpler proofs of some known results for
trace-class Weyl operators; the latter play an essential role in quantum
mechanics. We then apply time-frequency analysis techniques to prove a phase
space version of the KLM condition; the main tools are Gabor frames and the
Wigner formalism. Finally, discrete approximations of the KLM condition, which
are tractable numerically, are provided.

\end{abstract}

\section{Introduction}

The characterization of positivity properties for trace class operators on
$L^{2}(\mathbb{R}^{n})$ is an important topic, not only because it is an
interesting mathematical problem which still is largely open, but also because
of its potential applications to quantum mechanics and even cosmology. It is a
notoriously difficult part of functional analysis which has been tackled by
many authors but there have been few decisive advances since the pioneering
work of Kastler \cite{Kastler} and Loupias and Miracle-Sole
\cite{LouMiracle1,LouMiracle2}; see however Dias and Prata \cite{dipra}. While
some partial results have been obtained in connection with the study of
quantum density operators
\cite{Narcow2,Narcow3,Narcow,Narconnell,Narconnell88} when the operators under
consideration are expressed using the Weyl correspondence, very little is
known about them when they are given in terms of more general correspondences
(in \cite{sriwolf} Srinivas and Wolf give such a condition, but the necessity
statement is false as already noted by Mourgues \textit{et al}. \cite{mofean}%
). It seems in fact that the field, which was quite active in the late 1980s
hasn't much evolved since; the open questions remain open.

We shall tackle the problem using techniques which come from both quantum
mechanics and time-frequency analysis. The phase space representation mainly
employed is the $\eta$-cross-Wigner transform; for $\eta\in\mathbb{R}%
\setminus\{0\}$, this is defined by
\begin{equation}
W_{\eta}(\psi,\phi)(z)=\left(  \tfrac{1}{2\pi\eta}\right)  ^{n}\int%
_{\mathbb{R}^{n}}e^{-\frac{i}{\eta}p\cdot y}\psi(x+\tfrac{1}{2}y)\overline
{\phi(x-\tfrac{1}{2}y)}dy,\label{ww}%
\end{equation}
for $\psi,\phi\in L^{2}(\mathbb{R}^{n}).$ When $\eta=\hbar>0$, ($\hbar$ the Planck constant $h$ divided by $2\pi$) we
recapture the standard cross-Wigner function $W_{\hbar}(\psi,\phi)$, simply
denoted by $W(\psi,\phi)$. Setting $W_{\eta}(\psi,\psi)=W_{\eta}\psi$ and
$\lambda=\eta/\hbar$, we have%
\begin{equation}
W_{\eta}\psi(x,p)=|\lambda|^{-n}W\psi(x,\lambda^{-1}p).\label{scale1}%
\end{equation}
In particular, a change of $\eta$ into $-\eta$ yields%
\begin{equation}
W_{\eta}\psi=(-1)^{n}W_{-\eta}\overline{\psi}.\label{scale2}%
\end{equation}

Given a symbol $a\in\mathcal{S}^{\prime}(\mathbb{R}^{2n})$ (the space of
tempered distribution), the Weyl pseudodifferential operator $\widehat{A}%
_{\eta}^{\mathrm{W}}=\operatorname*{Op}_{\eta}^{\mathrm{W}}(a)$ is weakly
defined by
\begin{equation}
\langle\widehat{A}_{\eta}^{\mathrm{W}}\psi,\overline{\phi}\rangle=\langle
a,W_{\eta}(\psi,\phi)\rangle, \label{w1}%
\end{equation}
for all $\psi,\phi$ in the Schwartz class $\mathcal{S}(\mathbb{R}^{n})$
(Observe that $W_{\eta}(\psi,\phi)\in\mathcal{S}(\mathbb{R}^{2n})$). The
function $a$ is called the $\eta$-Weyl symbol of $\widehat{A}_{\eta
}^{\mathrm{W}}$.

Consider now a trace-class operator $\widehat{A}$ on $L^{2}(\mathbb{R}^{n})$
(see the definition in the subsequent Section \ref{sec2}). Then there exists
an orthonormal basis $(\psi_{j})$ for $L^{2}(\mathbb{R}^{n})$ and a sequence
$(\alpha_{j})\in\ell^{1}$ such that $\widehat{A}$ can be written as
\[
\widehat{A}=\sum_{j}\alpha_{j}\widehat{\Pi}_{j}%
\]
whit absolute convergence in $B(L^{2}(\mathbb{R}^{n}))$; here $\widehat{\Pi
}_{j}$ is the rank-one orthogonal projector of $L^{2}(\mathbb{R}^{n})$ onto
the one-dimensional subspace $\mathbb{C}\psi_{j}$ generated by $\psi_{j}$ (cf.
Lemma \ref{Lemma1}). It turns out that, under the additional assumption
$\widehat{A}$ to be self-adjoint, that $\widehat{A}$ can be represented as a
$\eta$-Weyl operator with corresponding symbol
\[
a=(2\pi\eta)^{n}\sum_{j}\alpha_{j}W_{\eta}\psi_{j}\in L^{2}(\mathbb{R}%
^{2n})\cap L^{\infty}(\mathbb{R}^{2n})
\]
(see Proposition \ref{Prop1}).

When $\widehat{A}$ is positive semidefinite and has trace equal\ to one, it is
called a \textit{density operator} (or density matrix, or stochastic, operator
in quantum mechanics); it is usually denoted by $\widehat{\rho}$. If the Weyl
symbol of $\widehat{\rho}$ is $a$, the function $\rho=(2\pi\eta)^{-n}a$ is
called the \textit{Wigner distribution} of $\widehat{\rho}$ in the quantum
mechanical literature. Given a trace class operator $\widehat{A}$ (positive or
not), the function
\begin{equation}
\rho=\sum_{j}\alpha_{j}W_{\eta}\psi_{j} \label{density}%
\end{equation}
is called the $\eta$-\emph{Wigner distribution} of $\widehat{A}$. (Observe
that $\rho\in L^{2}(\mathbb{R}^{2n})$).

We will henceforth assume that all the concerned operators are self-adjoint
and of trace class and denote them by $\widehat{\rho}$; such operators can
always be written as
\begin{equation}
\widehat{\rho}=\sum_{j}\alpha_{j}\widehat{\Pi}_{j}=(2\pi\eta)^{n}%
\operatorname*{Op}\nolimits_{\eta}^{\mathrm{W}}(\rho) \label{notation1}%
\end{equation}
the real function $\rho$ being given by formula (\ref{density}). We are going
to determine explicit necessary and sufficient conditions on $\rho$ ensuring
the positivity of $\widehat{\rho}$. To this goal, we will use the reduced
symplectic Fourier transform $F_{\Diamond}$, defined for $a\in\mathcal{S}%
(\mathbb{R}^{2n})$ by%
\begin{equation}
a_{\Diamond}(z)=F_{\Diamond}a(z)=\int_{\mathbb{R}^{2n}}e^{i\sigma(z,z^{\prime
})}a(z^{\prime})dz^{\prime} \label{adiam}%
\end{equation}
with $\sigma$ being the standard symplectic form. For $\eta\in\mathbb{R}%
\setminus\{0\}$, recall the symplectic $\eta$-Fourier transform%
\begin{equation}
a_{\sigma,\eta}(z)=F_{\sigma,\eta}a(z)=\left(  \tfrac{1}{2\pi\eta}\right)
^{n}\int_{\mathbb{R}^{2n}}e^{-\frac{i}{\eta}\sigma(z,z^{\prime})}a(z^{\prime
})dz^{\prime}. \label{w4}%
\end{equation}
Obviously $F_{\Diamond}$ is related to the symplectic $\eta$-Fourier transform
(\ref{w4}) by the formula
\begin{equation}
a_{\Diamond}(z)(z)=(2\pi\eta)^{n}a_{\sigma,\eta}(-\eta z). \label{diasig12}%
\end{equation}
With the notation (\ref{adiam}) Bochner's theorem \cite{Bochner,Katz} on
Fourier transforms of probability measures can be restated in the following way:

\begin{proposition}
[Bochner]\label{propbochner}A real function $\rho\in L^{1}(\mathbb{R}^{2n})$
is a probability density if and only if $\rho_{\Diamond}$ is continuous,
$\rho_{\Diamond}(0)=1$, and for all $z_{1},...,z_{N}\in\mathbb{R}^{2n}$ the
$N\times N$ matrix $\Lambda$ whose entries are the complex numbers
$\rho_{\Diamond}(z_{j}-z_{k})$ is positive semidefinite:
\begin{equation}
\Lambda=(\rho_{\Diamond}(z_{j}-z_{k}))_{1\leq j,k\leq N}\geq0. \label{bochner}%
\end{equation}

\end{proposition}

When condition (\ref{bochner}) is satisfied one says that $\rho_{\Diamond}$ is
of positive type. The notion of $\eta$-positivity, due to Kastler
\cite{Kastler}, generalizes this notion:

\begin{definition}
\label{defhpos}Let $a$ $\in L^{1}(\mathbb{R}^{2n})$ and $\eta\in
\mathbb{R}\setminus\{0\}$; we say that $a_{\Diamond}$\textit{\ }is of $\eta
$\textit{-positive type if for every integer }$N$ \textit{the }$N\times N$
matrix $\Lambda_{(N)}$ with entries
\[
\Lambda_{jk}=e^{-\frac{i\eta}{2}\sigma(z_{j},z_{k})}a_{\Diamond}(z_{j}-z_{k})
\]
is positive semidefinite for all choices of $(z_{1},z_{2},...,z_{N}
)\in(\mathbb{R}^{2n})^{N}$:
\begin{equation}
\Lambda_{(N)}=(\Lambda_{jk})_{1\leq j,k\leq N}\geq0. \label{fzjfzk}%
\end{equation}

\end{definition}

The condition (\ref{fzjfzk}) is equivalent to the polynomial inequalities
\begin{equation}
\sum_{1\leq j,k\leq N}\zeta_{j}\overline{\zeta_{k}}e^{-\frac{i\eta}{2}
\sigma(z_{j},z_{k})}a_{\Diamond}(z_{j}-z_{k})\geq0 \label{polynomial1}%
\end{equation}
for all $N\in\mathbb{N}$, $\zeta_{j},\zeta_{k}\in\mathbb{C}$, and $z_{j}%
,z_{k}\in\mathbb{R}^{2n}$.

It is easy to see that this implies $a_{\Diamond}(-z)=\overline{a_{\Diamond
}(z)}$ and therefore $a$ is real-valued.

\begin{remark}
\label{remeta}If $a$ is of $\eta$\textit{-positive type then it is also of }
$(-\eta)$\textit{-positive type.} This follows from the fact that the matrix
$(\overline{\Lambda_{jk}})_{1\leq j,k\leq N}$ is still positive semidefinite
and taking into account the equality $\overline{a_{\Diamond}(z)}=a_{\Diamond
}(-z)$.
\end{remark}

We first present a result originally due to Kastler \cite{Kastler}, and
Loupias and Miracle-Sole \cite{LouMiracle1,LouMiracle2} (the \textquotedblleft
KLM conditions\textquotedblright), who use the theory of $C^{\ast}$-algebras;
also see Parthasarathy \cite{partha1,partha2} and Parthasarathy and Schmidt
\cite{parthaschmidt}. The proof we give is simpler and is partially based on
the discussions in \cite{Narcow3,Narconnell,Werner}.

\begin{theorem}
[The KLM conditions] \label{Prop2} Let $\eta\in R\setminus\{0\}$ and let
$\widehat{A} =\operatorname*{Op}_{\eta}^{\mathrm{W}}(a)$ be a self-adjoint
trace-class  operator on $L^{2}(\mathbb{R}^{n})$ with symbol $a\in
L^{1}(\mathbb{R}^{2n})$.  We have $\widehat{A}\geq0$ if and only if the
conditions below hold:

(i) $a_{\Diamond}$ is continuous;

(ii) $a_{\Diamond}$ is of $\eta$\textit{-positive type}.
\end{theorem}

The KLM conditions are difficult to use in practice since they involve the
simultaneous verification of an uncountable set of conditions. We are going to
prove that they can be replaced with a countable set of conditions in phase
space. The key idea from time-frequency analysis is to use Gabor frames.

\begin{definition}
\label{ee0} Given a lattice $\Lambda$ in $\mathbb{R}^{2n}$ and a non-zero
function $g\in L^{2}(\mathbb{R}^{n})$, the system
\[
\mathcal{G}(g,\Lambda)=\{T(\lambda)g(x)=e^{i(\lambda_{2}x-\frac{1}{2}
\lambda_{1}\lambda_{2})}g(x-\lambda_{1}),\,\,\lambda=(\lambda_{1},\lambda
_{2})\in\Lambda\}
\]
is called a Gabor frame or Weyl-Heisenberg frame if it is a frame for
$L^{2}(\mathbb{R}^{n})$, that is there exist constants $0<A\leq B$ such that
\begin{equation}
A\Vert f\Vert_{2}^{2}\leq\sum_{z\in\Lambda}|\langle f,T(\lambda)g\rangle
|^{2}\leq B\Vert f\Vert_{2}^{2},\quad\forall f\in L^{2}(\mathbb{R} ^{n}).
\label{framedef}%
\end{equation}

\end{definition}

Hence, the $L^{2}$-norm of the function $f$ is equivalent to the $\ell^{2}$
norm of the sequence of its coefficients $\{\langle f,T_{1/(2\pi)}(\lambda)
g\rangle\}_{\lambda\in\Lambda}$ (cf. Section \ref{sec1} for more details).
Consider a Gabor frame ${\mathcal{G}}(\phi,\Lambda)$ for $L^{2}(\mathbb{R}%
^{n})$, with window $\phi\in L^{2}(\mathbb{R}^{n})$ and lattice $\Lambda
\in\mathbb{R}^{n}$. Let $a\in\mathcal{S}^{\prime}(\mathbb{R}^{2n})$ be a
symbol and denote by $a_{\lambda,\mu}$ its ``twisted" Gabor coefficient with
respect to the Gabor system $\mathcal{G}(W_{\eta}\phi,\Lambda\times\Lambda)$,
defined for $\lambda,\mu\in\Lambda\times\Lambda$ by
\begin{equation}
a_{\lambda,\mu}=\int_{\mathbb{R}^{2n}}e^{-\frac{i}{\eta}\sigma(z,\lambda-\mu
)}a(z)W_{\eta}\phi(z-\tfrac{1}{2}(\lambda+\mu))dz, \label{amunu}%
\end{equation}
where $W_{\eta}\psi=W_{\eta}(\psi,\psi)$ is the $\eta$-Wigner transform of
$\psi$.

Our main result characterizes the positivity of Hilbert--Schmidt operators
(and hence of trace class operators). It reads as follows:

\begin{theorem}
\label{ThmFabio}Let $a\in L^{2}(\mathbb{R}^{n})$ be real-valued and
$\widehat{A}_{\eta}=\operatorname*{Op}_{\eta}^{\mathrm{W}}(a)$.

(i) We have $\widehat{A}_{\eta}\geq0$ if and only if for every integer
$N\geq0$ the matrix $M_{(N)}$ with entries
\begin{equation}
M_{\lambda,\mu}=e^{-\frac{i}{2\eta}\sigma(\lambda,\mu)}a_{\lambda,\mu}\text{
\ , \ }|\lambda|,|\mu|\leq N \label{lm1}%
\end{equation}
is positive semidefinite.

(ii) One obtains an equivalent statement replacing the matrix $M_{(N)}$ with
the matrix $M^{\prime}_{(N)}$ where
\begin{equation}
M^{\prime}_{\lambda,\mu}=W_{\eta}(a,(W_{\eta}\phi)^{\vee})(\tfrac{1}
{4}(\lambda+\mu),\tfrac{1}{2}J(\mu-\lambda)) \label{lm2}%
\end{equation}
with $(W_{\eta}\phi)^{\vee}(z)=W_{\eta}\phi(-z)$.
\end{theorem}

The conditions in Theorem \ref{ThmFabio} only involve a countable set of
matrices, as opposed to the KLM ones. In addition, they are
\emph{well-organized} because the matrix of size $N$ is a submatrix of that
of  size $N+1$.

The KLM conditions can be recaptured by an averaging procedure from the ones
in Theorem \ref{ThmFabio}. To show this claim, we make use of another
well-known time-frequency representation: the short-time Fourier transform
(STFT). Precisely, for a given function $g\in\mathcal{S}(\mathbb{R}%
^{n})\setminus\{0\}$ (called window), the STFT $V_{g} f$ of a distribution
$f\in\mathcal{S^{\prime}}(\mathbb{R}^{n})$ is defined by
\begin{equation}
\label{STFT}V_{g} f(x,p)=\int_{\mathbb{R}^{n}}e^{- i p\cdot y} f(y)
\overline{g(y-x)}\,dy,\quad(x,p)\in\mathbb{R}^{2n}.
\end{equation}
Let $\phi_{0}(x)=(\pi\eta)^{-n/4}e^{-|x|^{2}/2\eta}$ be the standard Gaussian
and $\phi_{\nu}=T(\nu)\phi_{0}$, $\nu\in\mathbb{R}^{2n}$. We shall consider
the STFT $V_{W\phi_{\nu}}a$, with window given by the Wigner function
$W\phi_{\nu}$ and symbol $a$. Then we establish the following connection:

\begin{theorem}
\label{Theorem2}Let $a\in L^{1}(\mathbb{R}^{2n})$ and $\lambda,\mu
\in\mathbb{R}^{2n}$. We set
\begin{align}
M^{(KLM)}_{\lambda,\mu}  &  =e^{-\frac{i}{2\eta}\sigma(\lambda,\mu)}%
a_{\sigma,\eta}(\lambda-\mu)\nonumber\\
M_{\lambda,\mu}^{\phi_{\nu}}  &  =e^{-\frac{i}{2\eta}\sigma(\lambda,\mu
)}V_{W\phi_{\nu}}a(\tfrac{1}{2}(\lambda+\mu),J(\mu-\lambda)). \label{formula2}%
\end{align}
We have
\begin{equation}
\label{formula3}M^{(KLM)}_{\lambda,\mu}=(2\pi\eta)^{-n} \int_{\mathbb{R}^{2n}%
}M_{\lambda,\mu}^{\phi_{\nu}}\, d\nu.
\end{equation}

\end{theorem}

If the symbol $a\in L^{1}(\mathbb{R}^{n}) \cap L^{2}(\mathbb{R}^{n})$ and
choosing the lattice $\Lambda$ such that $\mathcal{G}(\phi_{0},\Lambda)$ is a
Gabor frame for $L^{2}(\mathbb{R}^{n})$, we obtain the following consequence:
\emph{If the matrix $(M_{\lambda,\mu}^{\phi_{0}})_{\lambda,\mu\in
\Lambda,|\lambda|,|\mu|\leq N}$ is positive semidefinite for every $N$, then
so is the matrix $(M^{(KLM)}_{\lambda,\mu})_{\lambda,\mu\in\Lambda
,|\lambda|,|\mu|\leq N}$} (cf. Corollary \ref{coro1}).

Finally, if the symbol $a$ is as before and and $\widehat{A}_{\eta
}=\operatorname*{Op}_{\eta}^{\mathrm{W}}(a)\geq0$, then for every finite
subset $S\subset\mathbb{R}^{2n}$ the matrix $(M^{(KLM)}_{\lambda,\mu
})_{\lambda,\mu\in S}$ is positive semidefinite. That is, the KLM conditions
hold (see Corollary \ref{coro2}).

The paper is organized as follows:

\begin{itemize}
\item In Section \ref{sec1} we briefly recall the main definitions and
properties of the Wigner--Weyl--Moyal formalism.

\item In Section \ref{sec2} we discuss the notion of positivity for trace
class operators; we also prove a continuous version of the positivity theorem
using the machinery of Hilbert--Schmidt operators.

\item In Section \ref{secKLM} we characterize positivity using the
Kastler--Loupias--Miracle-Sole (KLM) conditions of which we give a simple
proof. We give a complete description of trace class operators with Gaussian
Weyl symbols using methods which simplify and put on a rigorous footing older
results found in the physical literature.

\item In Section \ref{Secfabio} we show that the KLM conditions, which form an
uncountable set of conditions can be replaced with a set of countable
conditions involving the Wigner function. We thereafter study the notion of
\textquotedblleft almost positivity\textquotedblright\ which is an useful
approximation of the notion of positivity which can be easily implemented numerically.
\end{itemize}

\begin{notation}
We denote by $z=(x,p)$ the generic element of $\mathbb{R}^{2n}\equiv
\mathbb{R}^{n}\times\mathbb{R}^{n}$. Equipping $\mathbb{R}^{2n}$ with the
symplectic form $\sigma=\sum_{j}dp_{j}\wedge dx_{j}$ we denote by
$\operatorname*{Sp}(n)$ the symplectic group of $(\mathbb{R}^{2n},\sigma)$ and
by $\operatorname*{Mp}(n)$ the corresponding metaplectic group. $J=%
\begin{pmatrix}
0_{n\times n} & I_{n\times n}\\
-I_{n\times n} & 0_{n\times n}%
\end{pmatrix}
$ is the standard symplectic matrix, and we have $\sigma(z,z^{\prime})=Jz\cdot
z^{\prime}$. The $L^{2}$-scalar product is given by
\[
(\psi|\phi)_{L^{2}}=\int_{\mathbb{R}^{n}}\psi(x)\overline{\phi(x)}dx.
\]
The distributional pairing between $\psi\in\mathcal{S}^{\prime}(\mathbb{R}%
^{m})$ and $\phi\in\mathcal{S}(\mathbb{R}^{m})$ is denoted by $\langle
\psi,\phi\rangle$ regardless of the dimension $m$.

For $A,B\in \mathrm{GL}(m)$, we use the notation $A\backsim B$ to denote the
equality of two square matrices $A,B$ of same size $m\times m$ up to
conjugation: $A\backsim B$ if and only if there exists $C\in \mathrm{GL}(m)$
such that $A=C^{-1}BC$.
\end{notation}

\section{Weyl Operators and Gabor frames}

\label{sec1}

\subsection{The Weyl--Wigner formalism}

In what follows $\eta$ denotes a real parameter different from zero.

Given a symbol $a\in\mathcal{S}^{\prime}(\mathbb{R}^{2n})$ the Weyl
pseudodifferential operator $\widehat{A}_{\eta}^{\mathrm{W}}%
=\operatorname*{Op}_{\eta}^{\mathrm{W}}(a)$ is defined in \eqref{w1}, whereas
the $\eta$-cross-Wigner transform $W_{\eta}(\psi,\phi)$ is recalled in \eqref{ww}.

The operator $T_{\eta}(z)$ is Heisenberg's $\eta$-displacement operator%
\begin{equation}
T_{\eta}(z_{0})\psi(x)=e^{\frac{i}{\eta}(p_{0}x-\frac{1}{2}p_{0}x_{0})}%
\psi(x-x_{0}) \label{w3}%
\end{equation}
(see \cite{Birk,Birkbis}). The $\eta$-cross-ambiguity transform is defined by
\begin{equation}
\operatorname*{Amb}\nolimits_{\eta}(\psi,\phi)(z)=\left(  \tfrac{1}{2\pi\eta
}\right)  ^{n}(\psi|T_{\eta}(z)\phi)_{L^{2}}; \label{defamb}%
\end{equation}
we have \cite{Folland,Birkbis} the relation%
\begin{equation}
\operatorname*{Amb}\nolimits_{\eta}(\psi,\phi)=F_{\sigma,\eta}W_{\eta}%
(\psi,\phi), \label{w5}%
\end{equation}
where $F_{\sigma,\eta}$ is the symplectic $\eta$-Fourier transform already
recalled in \eqref{w4}. The functions $W_{\eta}\psi=W_{\eta}(\psi,\psi)$ and
$\operatorname*{Amb}\nolimits_{\eta}\psi=\operatorname*{Amb}\nolimits_{\eta
}(\psi,\psi)$ are called, respectively, the $\eta$-Wigner and $\eta$-ambiguity
transforms. The explicit expression of the $\eta$-Wigner transform is already
given in \eqref{ww}, whereas  the $\eta$-ambiguity transform is defined by
\begin{equation}
\operatorname*{Amb}\nolimits_{\eta}(\psi,\phi)(z) =\left(  \tfrac{1}{2\pi\eta
}\right)  ^{n}\int_{\mathbb{R}^{n}}e^{-\tfrac{i}{\eta}p\cdot y}\psi
(y+\tfrac{1}{2}x)\overline{\phi(y-\tfrac{1}{2}x)}dy. \label{wa}%
\end{equation}

Let $\widehat{A}_{\eta}^{\mathrm{W}}=\operatorname*{Op}_{\eta}^{\mathrm{W}%
}(a)$ and $\widehat{B}_{\eta}^{\mathrm{W}}=\operatorname*{Op}_{\eta
}^{\mathrm{W}}(b)$ and assume that $\widehat{A}_{\eta}^{\mathrm{W}}%
\widehat{B}_{\eta}^{\mathrm{W}}$ is defined on some subspace of $L^{2}%
(\mathbb{R}^{n})$; then the twisted symbol $c_{\sigma,\eta}$ of $\widehat{C}%
_{\eta}^{\mathrm{W}}=\widehat{A}_{\eta}^{\mathrm{W}}\widehat{B}_{\eta
}^{\mathrm{W}}$ is given by the \textquotedblleft twisted
convolution\textquotedblright\ \cite{Folland,Birkbis} $c_{\sigma,\eta
}=a_{\sigma,\eta}\ast_{\eta}b_{\sigma,\eta}$ defined by
\begin{equation}
(a_{\sigma,\eta}\star_{\eta}b_{\sigma,\eta})(z)=\left(  \tfrac{1}{2\pi\eta
}\right)  ^{n}\int_{\mathbb{R}^{2n}}e^{\frac{i}{2\eta}\sigma(z,z^{\prime}%
)}a_{\sigma,\eta}(z-z^{\prime})b_{\sigma,\eta}(z^{\prime})dz^{\prime}.
\label{twist1}%
\end{equation}
Alternatively, the symbol $c$ is given by the \textquotedblleft twisted
product\textquotedblright\ $c=a\times_{\hbar}b$ where%

\begin{equation}
(a\times_{\eta}b)(z)=\left(  \tfrac{1}{4\pi\eta}\right)  ^{2n}\int%
_{\mathbb{R}^{2n}}e^{\frac{i}{2\eta}\sigma(z^{\prime},z^{\prime\prime}%
)}a(z+\tfrac{1}{2}z^{\prime})b(z-\tfrac{1}{2}z^{\prime\prime})dz^{\prime
}dz^{\prime\prime}. \label{twist2}%
\end{equation}

An important property of the $\eta$-Wigner transform is that it satisfies the
\textquotedblleft marginal properties\textquotedblright%
\begin{equation}
\int_{\mathbb{R}^{n}}W_{\eta}\psi(z)dx=|F_{\eta}\psi(p)|^{2}\text{ \ , \ }%
\int_{\mathbb{R}^{n}}W_{\eta}\psi(z)dp=|\psi(x)|^{2},\label{marginal}%
\end{equation}
the first for every function $\psi\in L^{1}(\mathbb{R}^{n})\cap L^{2}%
(\mathbb{R}^{n})$, the second for every function $\psi\in L^{2}(\mathbb{R}%
^{n})$ such that $\hat{\psi}\in L^{1}(\mathbb{R}^{n})$;
here%
\begin{equation}
F_{\eta}\psi(p)=\left(  \tfrac{1}{2\pi|\eta|}\right)  ^{n/2}\int%
_{\mathbb{R}^{n}}e^{-\frac{i}{\eta}px}\psi(x)dx\label{Feta}%
\end{equation}
is the $\eta$-Fourier transform (see \cite{Folland,Springer}). Notice that
$F_{\eta}\psi$ and $F_{-\eta}\psi$ are related by the trivial formula%
\begin{equation}
F_{-\eta}\psi=(-1)^{n}\overline{F_{\eta}\overline{\psi}}.\label{fmineta}%
\end{equation}
It follows that $F_{\eta}$ extends into a topological unitary automorphism of
$L^{2}(\mathbb{R}^{n})$ for all values of $\eta\neq0$.

An important equality satisfied by the $\eta$-Wigner function is Moyal's
identity\footnote{It is sometimes also called the \textquotedblleft
orthogonality relation\textquotedblright\ for the Wigner function.}:

\begin{lemma}
Let $(\psi,\phi)\in L^{2}(\mathbb{R}^{n})\times L^{2}(\mathbb{R}^{n})$ and
$\eta\in\mathbb{R}\setminus\{0\}$. The function $W_{\eta}\psi$ is real and we
have
\begin{equation}
||W_{\eta}\psi||_{L^{2}(\mathbb{R}^{2n})}^{2}=\int_{\mathbb{R}^{2n}}W_{\eta
}\psi(z)W_{\eta}\phi(z)dz=\left(  \tfrac{1}{2\pi|\eta|}\right)  ^{n}%
|(\psi|\phi)|^{2}. \label{Moyaleta}%
\end{equation}
In particular%
\begin{equation}
\int_{\mathbb{R}^{2n}}W_{\eta}\psi(z)^{2}dz=\left(  \tfrac{1}{2\pi|\eta
|}\right)  ^{n}||\psi||^{4}. \label{Moyal2eta}%
\end{equation}

\end{lemma}

\begin{proof}
It is a standard result \cite{Folland,Gro} that (\ref{Moyaleta}) holds for all
$\eta>0$. The case $\eta<0$ follows using formula (\ref{scale2}).
\end{proof}

In Section \ref{Secfabio} we will use some concepts from time-frequency
analysis. We recall here the most important issues.

A Gabor frame $\mathcal{G}(\phi,\Lambda)$ is defined in Definition
\eqref{ee0}. This implies that any function $f\in L^{2}(\mathbb{R}^{n})$ can
be represented as
\[
f=\sum_{\lambda\in\Lambda} c_{\lambda}T(\lambda)g,
\]
with unconditional convergence in $L^{2}(\mathbb{R}^{n})$ and with suitable
coefficients $(c_{\lambda})_{\lambda}\in\ell^{2}(\Lambda)$.

A time-frequency representation closely related to the Wigner function is the
short-time Fourier transform (STFT), whose definition is in formula
\eqref{STFT}. Using this representation, we can define the Sj\"ostrand class
or modulation space $M^{\infty,1}_{v_{s}}$ \cite{F1,wiener30} in terms of the
decay of the STFS as follows. For $s\geq0$, consider the weight function
$v_{s}(z)=\langle z\rangle^{s}=(1+|z|^{2})^{s/2}$, $z\in\mathbb{R}^{2n}$,
then
\[
M^{\infty,1}_{v_{s}}(\mathbb{R}^{n})=\{f\in\mathcal{S^{\prime}}(\mathbb{R}%
^{n}): \|f\|_{M^{\infty,1}_{v_{s}}}:=\int_{\mathbb{R}^{n}}\sup_{x\in
\mathbb{R}^{n}}|V_{g} f(x,p)|v_{s}(x,p)\,dp<\infty\}.
\]
It can be shown that $\|f\|_{M^{\infty,1}_{v_{s}}}$ is a norm on $M^{\infty
,1}_{v_{s}}(\mathbb{R}^{n})$, independent of the window function
$g\in\mathcal{S}(\mathbb{R}^{n})$ (different windows yield equivalent norms).
Moreover $M^{\infty,1}_{v_{s}}(\mathbb{R}^{n})$ is a Banach space. For $s=0$
we simply write $M^{\infty,1}(\mathbb{R}^{n})$ in place of $M^{\infty
,1}_{v_{s}}(\mathbb{R}^{n})$.

Generally, by measuring the decay of the STFT by means of the mixed-normed
spaces $L_{v_{s}}^{p,q}(\mathbb{R}^{2n})$, one can define a scale of Banach
spaces known as modulation spaces. Here we will make use only of the so-called
Feichtinger's algebra (unweighted case $s=0$)
\[
M^{1}(\mathbb{R}^{n})=\{f\in\mathcal{S^{\prime}}(\mathbb{R}^{2n}):\Vert
f\Vert_{M^{1}}=\Vert V_{g}f\Vert_{L^{1}(\mathbb{R}^{2n})}<\infty\}.
\]
Notice that in Section \ref{Secfabio} we will work with spaces of symbols,
hence the dimension $n$ of the space is replaced by $2n$.

\section{The Positivity of Trace-Class Weyl Operators\label{sec2}}

Trace class operators play an essential role in quantum mechanics. A positive
semidefinite self-adjoint operator with unit trace is called a \emph{density
operator }(or \emph{density matrix} in the physical literature). Density
operators represent (and are usually identified with) the mixed quantum states
corresponding to statistical mixtures of quantum pure states.

\subsection{Trace class operators}

A bounded linear operator $\widehat{A}$ on $L^{2}(\mathbb{R}^{n})$ is of trace
class if for one (and hence every) orthonormal basis $(\psi_{j})_{j}$ of
$L^{2}(\mathbb{R}^{n})$ its modulus $|\widehat{A}|=(\widehat{A}^{\ast
}\widehat{A})^{1/2}$ satisfies
\begin{equation}
\sum_{j}(|\widehat{A}|\psi_{j}|\psi_{j})_{L^{2}}<\infty; \label{tr1}%
\end{equation}
the trace of $\widehat{A}$ is then, by definition, given by the absolutely
convergent series $\operatorname{Tr}(\widehat{A})=\sum_{j}(\widehat{A}\psi
_{j}|\psi_{j})_{L^{2}}$whose value is independent of the choice of the
orthonormal basis $(\psi_{j})_{j}$.

Trace class operators form a two-side ideal $\mathcal{L}_{1}(L^{2}%
(\mathbb{R}^{n}))$ in the algebra $\mathcal{L}(L^{2}(\mathbb{R}^{n}))$ of all
bounded linear operators on $L^{2}(\mathbb{R}^{n})$.

An operator $\widehat{A}\in\mathcal{L}(L^{2}(\mathbb{R}^{n}))$ is a
Hilbert--Schmidt operator if and only if there exists an orthonormal basis
$(\psi_{j})$ such that
\[
\sum_{j}(\widehat{A}\psi_{j}|\widehat{A}\psi_{j})_{L^{2}}<\infty.
\]
Hilbert--Schmidt operators form a two-sided ideal $\mathcal{L}_{2}%
(L^{2}(\mathbb{R}^{n}))$ in $\mathcal{L}(L^{2}(\mathbb{R}^{n}))$ and we have
$\mathcal{L}_{1}(L^{2}(\mathbb{R}^{n}))\subset\mathcal{L}_{2}(L^{2}%
(\mathbb{R}^{n}))$. A trace class operator can always be written (non
uniquely) as the product of two Hilbert--Schmidt operators (and is hence
compact). In particular, a positive (and hence self-adjoint) trace-class
operator can always be written in the form $\widehat{A}=\widehat{B}^{\ast
}\widehat{B}$ where $\widehat{B}$ is a Hilbert--Schmidt operator. One proves
(\cite{blabru}, \S 22.4, also \cite{Birk,Birkbis}) using the spectral theorem
for compact operators that for every trace-class operator on $L^{2}%
(\mathbb{R}^{n})$ there exists a sequence $(\alpha_{j})_{j}\in\ell
^{1}(\mathbb{N})$ and orthonormal bases $(\psi_{j})_{j}$ and $(\phi_{j})_{j}$
of $L^{2}(\mathbb{R}^{n})$ (indexed by the same set) such that for $\psi\in
L^{2}(\mathbb{R}^{n})$
\begin{equation}
\widehat{A}\psi=\sum_{j}\alpha_{j}(\psi|\psi_{j})_{L^{2}}\phi_{j}; \label{tr3}%
\end{equation}
conversely the formula above defines a trace-class operator on $L^{2}%
(\mathbb{R}^{n})$.
Observe that the series in \eqref{tr3} is absolutely convergent in
$L^{2}(\mathbb{R}^{n})$ (hence unconditionally convergent), since
\[
\sum_{j}\| \alpha_{j}(\psi|\psi_{j})_{L^{2}}\phi_{j}\|=\sum_{j}|\alpha_{j}|
|(\psi|\psi_{j})_{L^{2}}| \|\phi_{j}\|\newline\leq\sum_{j}|\alpha_{j}|\|\psi\|
\|\psi_{j}\|=\sum_{j}|\alpha_{j}| \|\psi\|.
\]
One verifies that the adjoint $\widehat{A}^{\ast}$ (which is also of trace
class) is given by%
\begin{equation}
\widehat{A}^{\ast}\psi=\sum_{j}\overline{\alpha_{j}}(\psi|\phi_{j})_{L^{2}%
}\psi_{j}. \label{tr4}%
\end{equation}
where the series is absolutely convergent in $L^{2}(\mathbb{R}^{n})$.

The following issue is an easy consequence of the spectral theorem for compact
self-adjoint operators.

\begin{lemma}
\label{Lemma1}Let $\widehat{A}$ be a trace-class operator on $L^{2}%
(\mathbb{R}^{n})$.

(i) If $\widehat{A}$ is self-adjoint there exists a real sequence $(\alpha
_{j})_{j}\in\ell^{1}(\mathbb{N})$ and an orthonormal basis $(\psi_{j})_{j}$ of
$L^{2}(\mathbb{R}^{n})$ such that%
\begin{equation}
\widehat{A}\psi=\sum_{j}\alpha_{j}(\psi|\psi_{j})_{L^{2}}\psi_{j} \label{tr5}%
\end{equation}
for every $\psi\in L^{2}(\mathbb{R}^{n})$,
with absolute convergence in $L^{2}(\mathbb{R}^{n})$;

(ii) if $\widehat{A}\geq0$ then (\ref{tr5}) holds with $\alpha_{j}\geq0$ for
all $j$.
\end{lemma}


Formula (\ref{tr5}) can be rewritten for short as%
\begin{equation}
\widehat{A}=\sum_{j}\alpha_{j}\widehat{\Pi}_{j} \label{tr6}%
\end{equation}
where $\widehat{\Pi}_{j}$ is a rank-one projector, namely the orthogonal
projection operator of $L^{2}(\mathbb{R}^{n})$ onto the one-dimensional
subspace $\mathbb{C}\psi_{j}$ generated by $\psi_{j}$.

Notice that the series \eqref{tr6} is absolutely convergent in $B(L^{2}%
(\mathbb{R}^{n}))$. Indeed, $\|\widehat{\Pi}_{j}\|_{B(L^{2})}=1$ for every $j$
and
\[
\sum_{j} \|\alpha_{j}\widehat{\Pi}_{j}\|_{B(L^{2})}\leq\sum_{j} |\alpha
_{j}|<\infty.
\]

\begin{proposition}
\label{Prop1} Let $\widehat{A}$ be a self-adjoint trace-class operator on
$L^{2}(\mathbb{R}^{n})$ as in \eqref{tr5} and $\eta>0$.

(i) The $\eta$-Weyl symbol $a$ of $\widehat{A}$ is given by
\begin{equation}
a=(2\pi\eta)^{n}\sum_{j}\alpha_{j}W_{\eta}\psi_{j} \label{tr7}%
\end{equation}
where the series converges absolutely in $L^{2}(\mathbb{R}^{2n})$ ;

(ii) The twisted symbol $a_{\sigma,\eta}$ is given by%
\begin{equation}
a_{\sigma,\eta}=(2\pi\eta)^{n}\sum_{j}\alpha_{j}\operatorname*{Amb}%
\nolimits_{\eta}\psi_{j}. \label{tr8}%
\end{equation}
with absolute convergence in $L^{2}(\mathbb{R}^{2n})$.

(iii) In particular, the symbols $a$ and $a_{\sigma,\eta}$ are in
$L^{2}(\mathbb{R}^{2n})\cap L^{\infty}(\mathbb{R}^{2n})$.
\end{proposition}

\begin{proof}
The distributional kernel of the orthogonal projection $\widehat{\Pi}_{j}$ is
$K_{j}=\psi_{j}\otimes\overline{\psi_{j}}$ hence the Weyl symbol $a_{j}$ of
$\widehat{\Pi}_{j}$ is given by the usual formula
\[
a_{j}(z)=\int_{\mathbb{R}^{n}}e^{-\frac{i}{\eta}py}K_{j}(x+\tfrac{1}%
{2}y,x-\tfrac{1}{2}y)dy=(2\pi\eta)^{n}W_{\eta}\psi_{j}(z).
\]
the series \eqref{tr6} being absolutely convergent in $B(L^{2}(\mathbb{R}%
^{n}))$. Moyal's identity
\[
\Vert\alpha_{j}W_{\eta}\psi_{j}\Vert_{2}=|\alpha_{j}|\left(  \tfrac{1}%
{2\pi\eta}\right)  ^{\frac{n}{2}}\Vert\psi_{j}\Vert_{2}^{2}=\left(  \tfrac
{1}{2\pi\eta}\right)  ^{\frac{n}{2}}|\alpha_{j}|
\]
and the assumption $(\alpha_{j})_{j}\in\ell^{1}(\mathbb{N})$ guarantee that
the series in (\ref{tr7}) is absolutely convergent in $L^{2}(\mathbb{R}^{2n})$
and we infer that the symbol $a$ is in $L^{2}(\mathbb{R}^{2n})$. Similarly, by
H\"{o}lder's inequality,
\[
|W_{\eta}\psi_{j}(z)|\leq\tfrac{2^{2n}}{(2\pi\eta)^{n}}\Vert\psi_{j}\Vert
_{2}^{2}%
\]
for all $z\in\mathbb{R}^{2n}$ so that
\[
\Vert\alpha_{j}W_{\eta}\psi_{j}\Vert_{\infty}\leq\left(  \tfrac{2}{\pi\eta
}\right)  ^{n}|\alpha_{j}|
\]
and the series in (\ref{tr7}) is absolutely convergent in $L^{\infty
}(\mathbb{R}^{2n})$, too. This proves our claim (iii) for the symbol $a$.
Formula (\ref{tr8}) follows since $W_{\eta}\psi_{j}$ and $\operatorname*{Amb}%
\nolimits_{\eta}\psi_{j}$ are symplectic $\eta$-Fourier transforms of each
other. Claim (iii) for $a_{\sigma,\eta}$ is obtained in a similar way.
\end{proof}

\begin{remark}
From Proposition \eqref{Prop1}, the functions $a$ and $a_{\sigma,\eta}$ are in
$L^{2}(\mathbb{R}^{2n})\cap L^{\infty}(\mathbb{R}^{2n})$. Notice that in
general the symbols $a$ and $a_{\sigma,\eta}$ are not in $L^{1}(\mathbb{R}%
^{2n})$. For example, choose $\widehat{A}=\widehat{{\Pi}}_{0}$, with
$\widehat{{\Pi}}_{0}$ the orthogonal projection onto a vector $\psi_{0}\in
L^{2}(\mathbb{R}^{n})\setminus(L^{1}(\mathbb{R}^{n})\cup\mathcal{F}%
L^{1}(\mathbb{R}^{n}))$.
\end{remark}

Recall that if $\widehat{A}$ is positive semidefinite and has trace equal\ to
one, it is called a \textit{density operator} and denoted by $\widehat{\rho}$.
We will from now on assume that all the concerned operators are self-adjoint
and of trace class, recalling from \eqref{notation1} that they can be written
as
\[
\widehat{\rho}=\sum_{j}\alpha_{j}\widehat{\Pi}_{j}=(2\pi\eta)^{n}%
\operatorname*{Op}\nolimits_{\eta}^{\mathrm{W}}(\rho)
\]
the real function $\rho$ being given by formula (\ref{density}). We are going
to determine explicit necessary and sufficient conditions on $\rho$ ensuring
the positivity of $\widehat{\rho}$. Let us first note the following result
which shows the sensitivity of density operators to changes in the value of
$\hbar$.

Let us now address the following question: for given $\psi\in L^{2}%
(\mathbb{R}^{n})$, can we find $\phi$ such that $W_{\eta}\phi=W\psi$ for
$\eta\neq\hbar$? The answer is negative:

\begin{proposition}
\label{Thm3}Let $\psi\in L^{1}(\mathbb{R}^{n})\cap L^{2}(\mathbb{R}%
^{n})\setminus\{0\}$ and $\eta\in\mathbb{R}\setminus\{0\}$, $\hbar>0$.

(i) There does not exist any $\phi\in L^{1}(\mathbb{R}^{n})\cap L^{2}%
(\mathbb{R}^{n})$ such that $W_{\eta}\phi=W\psi$ if $|\eta|\neq\hbar$.

(ii) Assume that there exist orthonormal systems $(\psi_{j})_{j\in\mathbb{N}}%
$, $(\phi_{j})_{j\in\mathbb{N}}$ of $L^{2}(\mathbb{R}^{n})$
and nonnegative sequences $\alpha=(\alpha_{j})_{j\in\mathbb{N}},\ \beta
=(\beta_{j})_{j\in\mathbb{N}}\in\ell^{1}(\mathbb{N})$ such that
\begin{equation}
\sum_{j}\alpha_{j}W_{\eta}\psi_{j}=\sum_{j}\beta_{j}W\phi_{j} \label{abcp}%
\end{equation}
Then we must have
\[
\hbar^{n}\Vert\alpha\Vert_{\ell^{2}}^{2}=|\eta|^{n}\Vert\beta\Vert_{\ell^{2}%
}^{2}.
\]

\end{proposition}

\begin{proof}
Observe that the series in (\ref{abcp}) are absolutely convergent in
$L^{2}(\mathbb{R}^{2n})$. (i) Assume that $W_{\eta}\phi=W\psi$; then, using
the first marginal property (\ref{marginal}),%
\[
|F_{\eta}\phi(p)|^{2}=\int_{\mathbb{R}^{n}}W_{\eta}\phi(x,p)dx=\int%
_{\mathbb{R}^{n}}W\psi(x,p)dx=|F_{\hbar}\psi(p)|^{2}%
\]
hence $\phi$ and $\psi$ must have the same $L^{2}$-norm: $||\phi||=||\psi||$
in view of Parseval's equality. On the other hand, using the Moyal identity
(\ref{Moyaleta}) for, respectively, $W\psi$ and $W_{\eta}\phi$,\ the equality
$W\psi=W_{\eta}\phi$ implies that%
\begin{align*}
\int_{\mathbb{R}^{2n}}W\psi(z)^{2}dz  &  =\left(  \tfrac{1}{2\pi\hbar}\right)
^{n}||\psi||^{4}\\
\int_{\mathbb{R}^{2n}}W_{\eta}\phi(z)^{2}dz  &  =\left(  \tfrac{1}{2\pi|\eta
|}\right)  ^{n}||\phi||^{4}%
\end{align*}
hence we must have $|\eta|=\hbar$.

(ii) Squaring both sides of \eqref{abcp} and integrating over $\mathbb{R}%
^{2n}$ we get, using again Moyal's identity and the orthonormality of the
vectors $\psi_{j}$ and $\phi_{j}$,
\[
\frac{1}{(2\pi|\eta|)^{n}}\sum_{j}\alpha_{j}^{2}=\frac{1}{(2\pi\hbar)^{n}}%
\sum_{j}\beta_{j}^{2},
\]
hence our claim.
\end{proof}

\begin{remark}
Assume in particular that $\beta_{1}=1$ and $\beta_{j}=0$ for $j\geq2$. Then
(ii) tells us that if $\sum_{j}\alpha_{j}W_{\eta}\psi_{j}=W\phi$ then we must
have $\hbar^{n}\Vert\alpha\Vert_{\ell^{2}}^{2}=|\eta|^{n}$. Assume that
$\sum_{j}\alpha_{j}=1$; then $\Vert\alpha\Vert_{\ell^{2}}^{2}\leq1$ hence we
must have $|\eta|\leq\hbar$.
\end{remark}

\section{Positive Trace Class Operators\label{secKLM}}

\subsection{A general positivity result for trace class operators}

We are now going to give an integral description of the positivity of a trace
class operator on $L^{2}(\mathbb{R}^{n})$ of which Theorem \ref{prop2} can be
viewed as a discretized version.

Let us begin by stating a general result:

\begin{lemma}
\label{Lemmatrab}Let $\widehat{A}=\operatorname*{Op}_{\eta}^{\mathrm{W}}(a)$
be a trace-class operator on $L^{2}(\mathbb{R}^{n})$, with $\eta>0$. We have
$\widehat{A}\geq0$ if and only $\operatorname*{Tr}(\widehat{A}\widehat{B}%
)\geq0$ for every positive trace class operator $\widehat{B}\in\mathcal{L}%
_{1}(L^{2}(\mathbb{R}^{n}))$.
\end{lemma}

\begin{proof}
Since $\mathcal{L}_{1}(L^{2}(\mathbb{R}^{n}))$ is itself an algebra the
product $\widehat{A}\widehat{B}$ is indeed of trace class so the condition
$\operatorname*{Tr}(\widehat{A}\widehat{B})\geq0$ makes sense; setting
$\widehat{B}=\operatorname*{Op}_{\eta}^{\mathrm{W}}(b)$ we have%
\[
b(z)=(2\pi\eta)^{n}\sum_{j}\beta_{j}W\psi_{j}%
\]
where $(\beta_{j})\in\ell^{1}(\mathbb{N})$ with $\beta_{j}\geq0$ and $\psi
_{j}$ an orthonormal basis for $L^{2}(\mathbb{R}^{n})$. Observing that trace
class operators are also Hilbert--Schmidt operators, we have \cite[Prop.
284]{Birkbis} since $a,b\in L^{2}(\mathbb{R}^{2n})$
\begin{equation}
\operatorname*{Tr}(\widehat{A}\widehat{B})=\int_{\mathbb{R}^{2n}}a(z)b(z)dz
\label{trab}%
\end{equation}
and hence
\[
\operatorname*{Tr}(\widehat{A}\widehat{B})=(2\pi\eta)^{n}\sum_{j}\beta_{j}%
\int_{\mathbb{R}^{2n}}a(z)W_{\eta}\psi_{j}(z)dz
\]
(the interchange of integral and series is justified by Fubini's Theorem).
Assume that $\operatorname*{Tr}(\widehat{A}\widehat{B})\geq0$. It is enough to
check the positivity of $\widehat{A}$ on unit vectors $\psi$ in $L^{2}%
(\mathbb{R}^{n})$. Choosing all the $\beta_{j}=0$ except $\beta_{1}$ and
setting $\psi_{1}=\psi$ we have
\begin{equation}
\int_{\mathbb{R}^{2n}}a(z)W\psi(z)dz\geq0; \label{ouap}%
\end{equation}
since we can choose $\psi\in L^{2}(\mathbb{R}^{n})$ arbitrarily, this means
that we have $\widehat{A}\geq0$. If, conversely, we have $\widehat{A}\geq0$
then (\ref{ouap}) holds for all $\psi_{j}$ hence $\operatorname*{Tr}%
(\widehat{A}\widehat{B})\geq0$.
\end{proof}

Let us now prove:

\begin{theorem}
\label{prop2}Let $\eta>0$ and $\widehat{A}=\operatorname*{Op}_{\eta
}^{\mathrm{W}}(a)$ be a trace-class operator. We have $\widehat{A}\geq0$ if
and only if
\begin{equation}
\int_{\mathbb{R}^{2n}}F_{\sigma,\eta}a(z)\left(  \int_{\mathbb{R}^{2n}%
}e^{-\frac{i}{2\eta}\sigma(z,z^{\prime})}c(z^{\prime}-z)\overline{c(z^{\prime
})}dz^{\prime}\right)  dz\geq0,
\end{equation}
for all $c\in L^{2}(\mathbb{R}^{2n})$.
\end{theorem}

\begin{proof}
In view of Lemma \ref{Lemmatrab} above we have $\widehat{A}\geq0$ if and only
if $\operatorname*{Tr}(\widehat{A}\widehat{B})\geq0$ for all positive
$\widehat{B}=\operatorname*{Op}_{\eta}^{\mathrm{W}}(b)\in\mathcal{L}_{1}%
(L^{2}(\mathbb{R}^{n}))$ that is (formula (\ref{trab}))%
\begin{equation}
\int_{\mathbb{R}^{2n}}a(z)b(z)dz=(a|b)_{L^{2}(\mathbb{R}^{2n})}\geq0.
\label{trab2}%
\end{equation}
(Recall that Weyl symbols of self-adjoint operators are real). Using
Plancherel's Theorem
\begin{equation}
(a|b)_{L^{2}(\mathbb{R}^{2n})}=\left(  \tfrac{1}{2\pi\eta}\right)
^{n}(F_{\sigma,\eta}a|F_{\sigma,\eta}b)_{L^{2}(\mathbb{R}^{2n})}.
\label{trab2bis}%
\end{equation}
Since $\widehat{B}\geq0$ there exists $\widehat{C}\in\mathcal{L}_{2}%
(L^{2}(\mathbb{R}^{n}))$ such that $\widehat{B}=\widehat{C}^{\ast}\widehat{C}$
and hence, setting $\widehat{C}=\operatorname*{Op}_{\eta}^{\mathrm{W}}(c)$
(recall that $\widehat{C}^{\ast}=\operatorname*{Op}_{\eta}^{\mathrm{W}}%
(\bar{c})$), by the composition formula for Weyl operators \eqref{twist1},
\begin{align*}
F_{\sigma,\eta}b(z)  &  =\left(  \tfrac{1}{2\pi\eta}\right)  ^{n}%
\int_{\mathbb{R}^{2n}}e^{\frac{i}{2\eta}\sigma(z,z^{\prime})}F_{\sigma,\eta
}\bar{c}(z-z^{\prime})F_{\sigma,\eta}c(z^{\prime})dz^{\prime}\\
&  =\left(  \tfrac{1}{2\pi\eta}\right)  ^{n}\int_{\mathbb{R}^{2n}}e^{\frac
{i}{2\eta}\sigma(z,z^{\prime})}\overline{F_{\sigma,\eta}{c}}(z^{\prime
}-z)F_{\sigma,\eta}c(z^{\prime})dz^{\prime}.
\end{align*}
Vice-versa, take any function $c\in L^{2}(\mathbb{R}^{2n})$ then
$\widehat{C}=\operatorname*{Op}_{\eta}^{\mathrm{W}}(c)$ is a Hilbert-Schmidt
operator and $\widehat{B}=\widehat{C}^{\ast}\widehat{C}$ is a positive
operator. Hence, using the fact that the operator $F_{\sigma,\eta}$ is a
topological automorphism of $L^{2}(\mathbb{R}^{2n})$, condition (\ref{trab2})
and \eqref{trab2bis} are equivalent to%
\begin{equation}
\int_{\mathbb{R}^{2n}}F_{\sigma,\eta}a(z)\left(  \int_{\mathbb{R}^{2n}%
}e^{-\frac{i}{2\eta}\sigma(z,z^{\prime})}c(z^{\prime}-z)\overline{c(z^{\prime
})}dz^{\prime}\right)  dz\geq0,
\end{equation}
for every $c\in L^{2}(\mathbb{R}^{2n})$, as claimed.
\end{proof}

\subsection{Proof of the KLM condition}

We are now going to prove the KLM conditions, that is Theorem \ref{Prop2}. We
will need the following classical result from linear algebra. It says that the
entrywise product of two positive semidefinite matrices is also positive semidefinite.

\begin{lemma}
[Schur]\label{lemmaschur}Let $M_{(N)}=(M_{jk})_{1\leq j,k\leq N}$ be the
Hadamard product $M_{(N)}^{\prime}\circ M_{(N)}^{\prime\prime}$ of the
matrices $M_{(N)}^{\prime}=(M_{jk}^{\prime})_{1\leq j,k\leq N}$ and
$M_{(N)}^{\prime\prime}=(M_{jk}^{\prime\prime})_{1\leq j,k\leq N}$:
$M_{(N)}=(M_{jk}^{\prime}M_{jk}^{\prime\prime})_{1\leq j,k\leq N}$. If
$M_{(N)}^{\prime}$ and $M_{(N)}^{\prime\prime}$ are positive semidefinite then
so is $M_{(N)}$.
\end{lemma}

\begin{proof}
See for instance Bapat \cite{Bapat}.
\end{proof}

We have now all the instruments to prove the KLM conditions.

\begin{proof}
[Proof of Theorem \ref{Prop2}]Let us first show that the conditions (i)--(ii)
are necessary. Assume that $\widehat{A}\geq0$. Since $a\in L^{1}%
(\mathbb{R}^{2n})$, the Riemann-Lebesgue Lemma gives that $a_{\Diamond}$ is
continuous. In view of Lemma \ref{Lemma1} and formula (\ref{tr7}) in
Proposition \ref{Prop1}\ we have
\begin{equation}
a=(2\pi\eta)^{n}\sum_{j}\alpha_{j}W_{\eta}\psi_{j}%
\end{equation}
for an orthonormal basis $\{\psi_{j}\}$ in $L^{2}(\mathbb{R}^{n})$, the
coefficients $\alpha_{j}$ being $\geq0$ and in $\ell^{1}(\mathbb{N})$. It is
thus sufficient to show that the Wigner transform $W_{\eta}\psi$ of an
arbitrary $\psi\in L^{2}(\mathbb{R}^{n})$ is of $\eta$-positive type. This
amounts to show that for all $(z_{1},...,z_{N})\in(\mathbb{R}^{2n})^{N}$ and
all $(\zeta_{1},...,\zeta_{N})\in\mathbb{C}^{N}$ we have
\begin{equation}
I_{N}(\psi)=\sum_{1\leq j,k\leq N}\zeta_{j}\overline{\zeta_{k}}e^{-\frac
{i}{2\eta}\sigma(z_{j},z_{k})}F_{\sigma,\eta}W_{\eta}\psi(z_{j}-z_{k})\geq0
\label{ineq121}%
\end{equation}
for every complex vector $(\zeta_{1},...,\zeta_{N})\in\mathbb{C}^{N}$ and
every sequence $(z_{1},...,z_{N})\in(\mathbb{R}^{2n})^{N}$. Since the $\eta
$-Wigner distribution $W_{\eta}\psi$ and the $\eta$-ambiguity function are
obtained from each other by the symplectic $\eta$-Fourier transform
$F_{\sigma,\eta}$ we have%
\[
I_{N}(\psi)=\sum_{1\leq j,k\leq N}\zeta_{j}\overline{\zeta_{k}}e^{-\frac
{i}{2\eta}\sigma(z_{j},z_{k})}\operatorname*{Amb}\nolimits_{\eta}\psi
(z_{j}-z_{k}).
\]
Let us prove that
\begin{equation}
I_{N}(\psi)=\left(  \tfrac{1}{2\pi\eta}\right)  ^{n}||\sum\nolimits_{1\leq
j\leq N}\zeta_{j}T_{\eta}(-z_{j})\psi||_{L^{2}}^{2}; \label{equin12}%
\end{equation}
the inequality (\ref{ineq121}) will follow. Taking into account the fact that
$T_{\eta}(-z_{k})^{\ast}=T_{\eta}(z_{k})$ and using the familiar relation
\cite{Folland,Birk,Birkbis}
\begin{equation}
T_{\eta}(z_{k}-z_{j})=e^{-\frac{i}{2\eta}\sigma(z_{j},z_{k})}T_{\eta}%
(z_{k})T_{\eta}(-z_{j}) \label{tzotzo12}%
\end{equation}
we have, expanding the square in the right-hand side of (\ref{equin12}),
\begin{align*}
||\sum_{1\leq j\leq N}\zeta_{j}T_{\eta}(-z_{j})\psi||_{L^{2}}^{2}  &
=\sum_{1\leq j,k\leq N}\zeta_{j}\overline{\zeta_{k}}(T_{\eta}(-z_{j}%
)\psi|T_{\eta}(-z_{k})\psi)_{L^{2}}\\
&  =\sum_{1\leq j,k\leq N}\zeta_{j}\overline{\zeta_{k}}(T_{\eta}(z_{k}%
)T_{\eta}(-z_{j})\psi|\psi)_{L^{2}}\\
&  =\sum_{1\leq j,k\leq N}\zeta_{j}\overline{\zeta_{k}}e^{-\tfrac{i}{2\eta
}\sigma(z_{j},z_{k})}(T_{\eta}(z_{k}-z_{j})\psi|\psi)_{L^{2}}\\
&  =\left(  2\pi\eta\right)  ^{n}\sum_{1\leq j,k\leq N}\zeta_{j}%
\overline{\zeta_{k}}e^{-\tfrac{i}{2\eta}\sigma(z_{j},z_{k})}%
\operatorname*{Amb}\nolimits_{\eta}\psi(z_{j}-z_{k})
\end{align*}
proving the equality (\ref{equin12}). Let us now show that, conversely, the
conditions (i) and (ii) are sufficient, \textit{i.e.} that they imply that
$(\widehat{A}\psi|\psi)_{L^{2}}\geq0$ for all $\psi\in L^{2}(\mathbb{R}^{n})$;
equivalently (see formula (\ref{w1}))
\begin{equation}
\int_{\mathbb{R}^{2n}}a(z)W_{\eta}\psi(z)dz\geq0 \label{integraltoprove}%
\end{equation}
for $\psi\in L^{2}(\mathbb{R}^{n})$. Let us set, as above,%
\begin{equation}
\Lambda_{jk}^{\prime}=e^{\frac{i}{2\eta}\sigma(z_{j},z_{k})}a_{\sigma,\eta
}(z_{j}-z_{k})
\end{equation}
where $z_{j}$ and $z_{k}$ are arbitrary elements of $\mathbb{R}^{2n}$. To say
that $a_{\sigma,\eta}$ is of $\eta$\textit{-}positive type means that the
matrix $\Lambda^{\prime}=(\Lambda_{jk}^{\prime})_{1\leq j,k\leq N}$ is
positive semidefinite; choosing $z_{k}=0$ and setting $z_{j}=z$ this means
that every matrix $(a_{\sigma,\eta}(z))_{1\leq j,k\leq N}$ is positive
semidefinite. Setting
\begin{align*}
\Gamma_{jk}  &  =e^{\frac{i}{2\eta}\sigma(z_{j},z_{k})}F_{\sigma,\eta}W_{\eta
}\psi(z_{j}-z_{k})\\
&  =e^{\frac{i}{2\eta}\sigma(z_{j},z_{k})}\operatorname*{Amb}\nolimits_{\eta
}\psi(z_{j}-z_{k})
\end{align*}
the matrix $\Gamma_{(N)}=(\Gamma_{jk})_{1\leq j,k\leq N}$ is positive
semidefinite. Let us now write
\[
M_{jk}=\operatorname*{Amb}\nolimits_{\eta}\psi(z_{j}-z_{k})a_{\sigma,\eta
}(z_{j}-z_{k});
\]
we claim that the matrix $M_{(N)}=(M_{jk})_{1\leq j,k\leq N}$ is positive
semidefinite. In fact, $M$ is the Hadamard product of the positive
semidefinite matrices $M_{(N)}^{\prime}=(M_{jk}^{\prime})_{1\leq j,k\leq N}$
and $M_{(N)}^{\prime\prime}=(M_{jk}^{\prime\prime})_{1\leq j,k\leq N}$ where%
\begin{align*}
M_{jk}^{\prime}  &  =e^{\frac{i}{2\eta}\sigma(z_{j},z_{k})}\operatorname*{Amb}%
\nolimits_{\eta}\psi(z_{j}-z_{k})\\
M_{jk}^{\prime\prime}  &  =e^{-\frac{i}{2\eta}\sigma(z_{j},z_{k})}%
a_{\sigma,\eta}(z_{j}-z_{k})\text{\ }%
\end{align*}
and Lemma \ref{lemmaschur} implies that $M_{(N)}$ is also positive
semidefinite. It follows from Bochner's theorem that the function $b$ defined
by
\[
b_{\sigma,\eta}(z)=\operatorname*{Amb}\nolimits_{\eta}\psi(z)a_{\sigma,\eta
}(-z)
\]
is a probability density hence $b(z)\geq0$ for all $z\in\mathbb{R}^{2n}$.
Integrating this equality we get, using the Plancherel formula for the
symplectic $\eta$-Fourier transform,
\begin{align*}
(2\pi\eta)^{n}b(0)  &  =\int_{\mathbb{R}^{2n}}\operatorname*{Amb}%
\nolimits_{\eta}\psi(z)a_{\sigma,\eta}(-z)dz\\
&  =\int_{\mathbb{R}^{2n}}W_{\eta}\psi(z)a(z)dz
\end{align*}
hence the inequality (\ref{integraltoprove}) since $b(0)\geq0$.
\end{proof}

\subsection{The\ Gaussian case}

Let $\Sigma$ be a positive symmetric (real) $2n\times2n$ matrix and consider
the Gaussian%
\begin{equation}
\rho(z)=(2\pi)^{-n}\sqrt{\det\Sigma^{-1}}e^{-\frac{1}{2}\Sigma^{-1}z^{2}}.
\label{Gauss}%
\end{equation}
Let us find for which values of $\eta$ the function $\rho$ is the $\eta
$-Wigner distribution of a density operator. Narcowich \cite{Narcow2} was the
first to address this question using techniques from harmonic analysis using
the approach in Kastler's paper \cite{Kastler}; we give here a new and simpler
proof using the multidimensional Hardy's uncertainty principle, which we state
in the following form:

\begin{lemma}
\label{LemmaHardy}Let $A$ and $B$ be two real positive definite matrices and
$\psi\in L^{2}(\mathbb{R}^{n})$, $\psi\neq0$. Assume that
\begin{equation}
|\psi(x)|\leq Ce^{-\tfrac{1}{2}Ax^{2}}\text{ \ and \ }|F_{\eta}\psi(p)|\leq
Ce^{-\tfrac{1}{2}Bp^{2}} \label{AB}%
\end{equation}
for a constant $C>0$. Then:

(i) The eigenvalues $\lambda_{j}$, $j=1,...,n$, of the matrix $AB$ are all
$\leq1/\eta^{2}$;

(ii) If $\lambda_{j}=1/\eta^{2}$ for all $j$, then $\psi(x)=ke^{-\frac{1}%
{2}Ax^{2}}$ for some constant $k$.
\end{lemma}

\begin{proof}
See de Gosson and Luef \cite{goluhardy}, de Gosson \cite{Birkbis}. The $\eta
$-Fourier transform $F_{\eta}\psi$ in the second inequality (\ref{AB}) is
given by \eqref{Feta}.
\end{proof}

We will also need the two following lemmas; the first is a positivity result.

\begin{lemma}
\label{Lemmaposex} If $R$ is a symmetric positive semidefinite $2n\times2n$
matrix, then%
\begin{equation}
P_{(N)}=\left(  Rz_{j}\cdot z_{k}\right)  _{1\leq j,k\leq N} \label{P}%
\end{equation}
is a symmetric positive semidefinite $N\times N$ matrix for all $z_{1}%
,...,z_{N}\in\mathbb{R}^{2n}$.
\end{lemma}

\begin{proof}
There exists a matrix $L$ such that $R=L^{\ast}L$ (Cholesky decomposition).
Denoting by $\langle z|z^{\prime}\rangle=z\cdot\overline{z^{\prime}}$ the
inner product on $\mathbb{C}^{2n}$ we have, since the $z_{j}$ are real
vectors,
\[
L^{\ast}z_{j}\cdot z_{k}=\langle L^{\ast}z_{j}|z_{k}\rangle=\langle
z_{j}|Lz_{k}\rangle=z_{j}\cdot\overline{Lz_{k}}%
\]
hence $Rz_{j}\cdot z_{k}=Lz_{j}\cdot\overline{Lz_{k}}$. It follows that for
all complex $\zeta_{j}$ we have
\[
\sum_{1\leq j,k\leq N}\zeta_{j}\overline{\zeta_{k}}Rz_{j}\cdot z_{k}%
=\sum_{1\leq j\leq N}\zeta_{j}Lz_{j}\overline{\left(  \sum_{1\leq j\leq
N}\zeta_{j}Lz_{j}\right)  }\geq0
\]
hence our claim.
\end{proof}

The second lemma is a well-known diagonalization result (Williamson's
symplectic diagonalization theorem \cite{Folland,Birk}):

\begin{lemma}
\label{Williamson}Let $\Sigma$ be a symmetric positive definite real
$2n\times2n$ matrix. There exists $S\in\operatorname*{Sp}(n)$ such that
$\Sigma=S^{T}DS$ where
\[
D=%
\begin{pmatrix}
\Lambda & 0\\
0 & \Lambda
\end{pmatrix}
\]
with $\Lambda=\operatorname*{diag}(\lambda_{1},...,\lambda_{n})$, the positive
numbers $\lambda_{j}$ being the symplectic eigenvalues of $\Sigma$ (that is,
$\pm i\lambda_{1},...,\pm i\lambda_{n}$ are the eigenvalues of $J\Sigma
\backsim\Sigma^{1/2}J\Sigma^{1/2}$).
\end{lemma}

\begin{proof}
See for instance\ \cite{Folland,Birk,Birkbis}.
\end{proof}

We now have the tools needed to give a complete characterization of Gaussian
$\eta$-Wigner distributions:

\begin{proposition}
Let $\eta\in R\setminus\{0\}$. The Gaussian function (\ref{Gauss}) is the
$\eta$-Wigner transform of a positive trace class operator if and only if
\begin{equation}
|\eta|\leq2\lambda_{\min} \label{etalambda}%
\end{equation}
where $\lambda_{\min}$ is the smallest symplectic eigenvalue of $\Sigma$;
equivalently the self-adjoint matrix $\Sigma+i\eta J$ is positive
semidefinite:
\begin{equation}
\Sigma+\frac{i\eta}{2}J\geq0. \label{sigmaj}%
\end{equation}

\end{proposition}

\begin{proof}
Let us first show that the conditions (\ref{etalambda}) and (\ref{sigmaj}) are
equivalent. Let $\Sigma=S^{T}DS$ be a symplectic diagonalization of $\Sigma$
(Lemma \ref{Williamson}). Since $S^{T}JS=J$ condition (\ref{sigmaj}) is
equivalent to $D+\frac{i\eta}{2}J\geq0$. Let $z=(x,p)$ be an eigenvector of
$D+\frac{i\eta}{2}J$; the corresponding eigenvalue $\lambda$ is real and
$\geq0$. The characteristic polynomial of $D+\frac{i\eta}{2}J$ is
\[
P(\lambda)=\det\left[  (\Lambda-\lambda I)^{2}-\tfrac{\eta^{2}}{4}I\right]
=P_{1}(\lambda)\cdot\cdot\cdot P_{n}(\lambda)
\]
where
\[
P_{j}(\lambda)=(\lambda_{j}-\lambda)^{2}-\tfrac{\eta^{2}}{4}%
\]
hence the eigenvalues $\lambda$ of $D+\frac{i\eta}{2}J$ are the numbers
$\lambda=\lambda_{j}\pm\frac{1}{2}\eta$; since $\lambda\geq0$ the condition
$D+\frac{i\eta}{2}J\geq0$ is equivalent to $\lambda_{j}\geq\sup\{\pm\frac
{1}{2}\eta\}=\frac{1}{2}|\eta|$ for all $j$, which is the condition
(\ref{etalambda}). Let us now show that the condition (\ref{etalambda}) is
necessary for the function
\begin{equation}
\rho(z)=(2\pi)^{-n}\sqrt{\det\Sigma^{-1}}e^{-\frac{1}{2}\Sigma^{-1}z^{2}}%
\end{equation}
to be $\eta$-Wigner transform of a positive trace class operator. Let
$\widehat{\rho}=(2\pi\eta)^{n}\operatorname*{Op}_{\eta}^{\mathrm{W}}(\rho)$
and set $a(z)=(2\pi\eta)^{n}\rho(z)$. Let $\widehat{S}$ $\in\operatorname*{Mp}%
(n)$; the operator $\widehat{\rho}$ is of trace class if only if
$\widehat{S}\widehat{\rho}\widehat{S}^{-1}$ is, in which case
$\operatorname{Tr}(\widehat{\rho})=\operatorname{Tr}(\widehat{S}\widehat{\rho
}\widehat{S}^{-1})$. Choose $\widehat{S}$ with projection $S\in
\operatorname*{Sp}(n)$ such that $\Sigma=S^{T}DS$ is a symplectic
diagonalization of $\Sigma$. This choice reduces the proof to the case
$\Sigma=D$, that is to%
\begin{equation}
\rho(z)=(2\pi)^{-n}(\det\Lambda^{-1})e^{-\frac{1}{2}(\Lambda^{-1}x^{2}%
+\Lambda^{-1}p^{2})}. \label{Gaussdiag}%
\end{equation}
Suppose now that $\widehat{\rho}$ is of trace class; then there exist an
orthonormal basis of functions $\psi_{j}\in L^{2}(\mathbb{R}^{n})$ ($1\leq
j\leq n$) such that
\[
\rho(z)=\sum_{j}\alpha_{j}W_{\eta}\psi_{j}(z)
\]
where the $\alpha_{j}\geq0$ sum up to one. Integrating with respect to the $p$
and $x$ variables, respectively, the marginal conditions satisfied by the
$\eta$-Wigner transform and formula (\ref{Gaussdiag}) imply that we have%
\begin{align*}
\sum_{j}\alpha_{j}|\psi_{j}(x)|^{2}  &  =(2\pi)^{-n/2}(\det\Lambda
)^{1/2}e^{-\frac{1}{2}\Lambda^{-1}x^{2}}\\
\sum_{j}\alpha_{j}|F_{\eta}\psi_{j}(p)|^{2}  &  =(2\pi)^{-n/2}(\det
\Lambda)^{1/2}e^{-\frac{1}{2}\Lambda^{-1}p^{2}}.
\end{align*}
In particular, since $\alpha_{j}\geq0$ for every $j=1,2,...,n,\dots$,%
\[
|\psi_{j}(x)|\leq C_{j}e^{-\frac{1}{4}\Lambda^{-1}x^{2}}\text{ \ , \ }%
|F_{\eta}\psi_{j}(p)|\leq C_{j}e^{-\frac{1}{4}\Lambda^{-1}p^{2}}%
\]
here, if $\alpha_{j}\not =0$, $C_{j}=(2\pi)^{-n/4}(\det\Lambda)^{1/4}%
/\alpha_{j}^{1/2}$. Applying Lemma \ref{LemmaHardy} with $A=B=\frac{\eta}%
{2}\Lambda^{-1}$ we must have $|\eta|\leq2\lambda_{j}$ for all $j=1,\dots n$,
which is condition (\ref{etalambda}); this establishes the sufficiency
statement. (iii) Let us finally show that, conversely, the condition
(\ref{sigmaj}) is sufficient. It is again no restriction to assume that
$\Sigma$ is the diagonal matrix $D=%
\begin{pmatrix}
\Lambda & 0\\
0 & \Lambda
\end{pmatrix}
$; the symplectic Fourier transform of $\rho$ is easily calculated and one
finds that $\rho_{\Diamond}(z)=e^{-\frac{1}{4}Dz^{2}}$. Let $\Lambda
_{(N)}=(\Lambda_{jk})_{1\leq j,k\leq N}$ with%
\[
\Lambda_{jk}=e^{-\frac{i\eta}{2}\sigma(z_{j},z_{k})}\rho_{\Diamond}%
(z_{j}-z_{k});
\]
a simple algebraic calculation shows that we have
\[
\Lambda_{jk}=e^{-\frac{1}{4}Dz_{j}^{2}}e^{\frac{1}{2}(D+i\eta J)z_{j}\cdot
z_{k}}e^{-\frac{1}{4}Dz_{k}^{2}}%
\]
and hence%
\[
\Lambda_{(N)}=\Delta_{(N)}\Gamma_{(N)}\Delta_{(N)}^{\ast}%
\]
where $\Delta_{(N)}=\operatorname*{diag}(e^{-\frac{1}{4}Dz_{1}^{2}%
},...,e^{-\frac{1}{4}Dz_{N}^{2}})$ and $\Gamma_{(N)}=(\Gamma_{jk})_{1\leq
j,k\leq N}$ with $\Gamma_{jk}=e^{\frac{1}{2}(D+i\eta J)z_{j}\cdot z_{k}}$. The
matrix $\Lambda_{(N)}$ is thus positive semidefinite if and only if
$\Gamma_{(N)}$ is, but this is the case in view of Lemma \ref{Lemmaposex}.
\end{proof}

\begin{remark}
Setting $2\lambda_{\min}=\hslash$ and writing $\Sigma$ in the block-matrix
form $%
\begin{pmatrix}
\Sigma_{xx} & \Sigma_{xp}\\
\Sigma_{px} & \Sigma_{pp}%
\end{pmatrix}
$ where $\Sigma_{xx}=(\sigma_{x_{j},x_{k}})_{1\leq j,k\leq n}$, $\Sigma
_{xp}=(\sigma_{x_{j},p_{k}})_{1\leq j,k\leq n}$ and so on, one shows
\cite{goluPR} that (\ref{sigmaj}) is equivalent to the generalized uncertainty
relations (\textquotedblleft Robertson--Schr\"{o}dinger
inequalities\textquotedblright, see \cite{goluPR})%
\begin{equation}
\sigma_{x_{j}}^{2}\sigma_{p_{j}}^{2}\geq\sigma_{x_{j},p_{j}}^{2}+\tfrac{1}%
{4}\hbar^{2} \label{RS}%
\end{equation}
where, for $\leq j\leq n$, the $\sigma_{x_{j}}^{2}=\sigma_{x_{j},x_{j}}$,
$\sigma_{p_{j}}^{2}=\sigma_{p_{j},p_{j}}$ are viewed as variances and the
$\sigma_{x_{j},p_{j}}^{2}$ as covariances.
\end{remark}

\section{The KLM Conditions in phase space\label{Secfabio}}

\subsection{The main result}

We now consider a Gabor (or Weyl--Heisenberg) frame $\mathcal{G}(\phi
,\Lambda)$ for $L^{2}(\mathbb{R}^{n})$, with window $\phi\in L^{2}%
(\mathbb{R}^{n})$ and lattice $\Lambda\in\mathbb{R}^{n}$ (cf.
\eqref{framedef}). Time-frequency analysis and the Wigner formalism are the
main ingredients for proving our main result.

\begin{proof}
[Proof of Theorem \ref{ThmFabio}](i) Since $\mathcal{G}(\phi,\Lambda)$ is a
Gabor frame, we can write
\begin{equation}
\label{gaborespansione}\psi=\sum_{\lambda\in\Lambda}c_{\lambda}T(\lambda)\phi
\end{equation}
for some $(c_{\lambda})\in\ell^{2}(\Lambda)$. Let us prove that
\begin{equation}
\label{aggiunta}\int_{\mathbb{R}^{2n}}a(z)W_{\eta}\psi(z)dz=\sum
\nolimits_{\lambda,\mu\in\Lambda}c_{\lambda}\overline{c_{\mu}}e^{-\frac
{i}{2\eta}\sigma(\lambda,\mu)}a_{\lambda,\mu}.
\end{equation}

In view of the sesquilinearity of the cross-Wigner transform and its
continuity as a map $L^{2}(\mathbb{R}^{n})\times L^{2}(\mathbb{R}^{n})\to
L^{2}(\mathbb{R}^{2n})$ we have%
\[
W_{\eta}\left(
{\textstyle\sum\nolimits_{\lambda\in\Lambda}}
c_{\lambda}T(\lambda)\phi\right)  =%
{\textstyle\sum\nolimits_{\lambda\in\Lambda}}
c_{\lambda}\overline{c_{\mu}}W_{\eta}(T(\lambda)\phi,T(\mu)\phi).
\]
Using the relation (formula (9.23) in \cite{Birkbis})%
\[
W_{\eta}(T(\lambda)\phi,T(\mu)\phi)=e^{-\frac{i}{2\eta}\sigma(\lambda,\mu
)}e^{-\frac{i}{\eta}\sigma(z,\lambda-\mu)}W_{\eta}\phi(z-\tfrac{1}{2}%
(\lambda+\mu))
\]
we obtain
\begin{align*}
\int_{\mathbb{R}^{2n}}  &  a(z)W_{\eta}\psi(z)dz\\
&  =c_{\lambda}\overline{c_{\mu}}e^{-\frac{i}{2\eta}\sigma(\lambda,\mu)}%
\int_{\mathbb{R}^{2n}}a(z)e^{-\frac{i}{\eta}\sigma(z,\lambda-\mu)}W_{\eta}%
\phi(z-\tfrac{1}{2}(\lambda+\mu))dz\\
&  =\sum\nolimits_{\lambda,\mu\in\Lambda}c_{\lambda}\overline{c_{\mu}%
}e^{-\frac{i}{2\eta}\sigma(\lambda,\mu)}a_{\lambda,\mu}.
\end{align*}
Suppose now $\widehat{A}_{\eta}\geq0$, that is
\begin{equation}
\int_{\mathbb{R}^{2n}}a(z)W_{\eta}\psi(z)dz\geq0 \label{apos1}%
\end{equation}
for every $\psi\in L^{2}(\mathbb{R}^{n})$. For any given sequence
$(c_{\lambda})_{\lambda\in\Lambda}$ with $c_{\lambda}=0$ for $|\lambda|>N$ we
take $\psi$ as in \eqref{gaborespansione} and apply \eqref{aggiunta}; we
obtain that the finite matrix in \eqref{lm1} is positive semidefinite.

Assume conversely that the matrix in \eqref{lm1} is positive semidefinite for
every $N$; then the right-hand side of \eqref{aggiunta} is nonnegative,
whenever the series converges (unconditionally). Now, every $\psi\in
L^{2}(\mathbb{R}^{n})$ has a Gabor expansion as in \eqref{gaborespansione},
with $c_{\lambda}=(\psi|T(\lambda)\gamma)_{L^{2}}$ in $\ell^{2}(\Lambda)$, for
some dual window $\gamma\in L^{2}(\mathbb{R}^{n})$. Hence from
\eqref{aggiunta} and \eqref{apos1} we deduce $\widehat{A}\geq0$.

(ii) The desired result follows from the following calculation: we have%
\[
M^{\prime}_{\lambda,\mu}=W_{\eta}(a,(W_{\eta}\phi)^{\vee})(\tfrac{1}%
{4}(\lambda+\mu),\tfrac{1}{2}J(\mu-\lambda))
\]
that is, by definition of the Wigner transform,%
\begin{multline*}
M^{\prime}_{\lambda,\mu}=\left(  \tfrac{1}{2\pi\eta}\right)  ^{2n}%
\int_{\mathbb{R}^{2n}}e^{-\frac{i}{2\eta}\sigma(\mu-\lambda,u)}a(\tfrac{1}%
{4}(\lambda+\mu)+\tfrac{1}{2}u)\\
\times W_{\eta}\phi(-\tfrac{1}{4}(\lambda+\mu)+\tfrac{1}{2}u)du;
\end{multline*}
setting $z=\tfrac{1}{4}(\lambda+\mu)+\tfrac{1}{2}u$ we have $u=2z-\tfrac{1}%
{2}(\lambda+\mu)$ and hence%
\[
M^{\prime}_{\lambda,\mu}=\left(  \tfrac{1}{2\pi\eta}\right)  ^{2n}2^{n}%
\int_{\mathbb{R}^{2n}}e^{-\frac{i}{2\eta}\sigma(\mu-\lambda,2z-\tfrac{1}%
{2}(\lambda+\mu))}a(z)W_{\eta}\phi(z-\tfrac{1}{2}(\lambda+\mu))dz.
\]
Using the bilinearity and antisymmetry of the symplectic form $\sigma$ we have%
\[
\sigma(\mu-\lambda,2z-\tfrac{1}{2}(\lambda+\mu))=2\sigma(z,\lambda-\mu
)+\sigma(\lambda,\mu)
\]
so that%
\[
M^{\prime}_{\lambda,\mu}=\left(  \tfrac{1}{2\pi\eta}\right)  ^{2n}%
2^{n}e^{-\frac{i}{2\eta}\sigma(\lambda,\mu)}\int_{\mathbb{R}^{2n}}e^{-\frac
{i}{\eta}\sigma(z,\lambda-\mu)}a(z)W_{\eta}\phi(z-\tfrac{1}{2}(\lambda
+\mu))dz
\]
that is $M^{\prime}_{\lambda,\mu}=2^{n}(2\pi\eta)^{-2n}M_{\lambda,\mu}$, hence
our claim.
\end{proof}

\begin{remark}
Let us observe that Theorem \ref{ThmFabio} extends to other classes of
symbols, essentially with the same proof. For example the results hold for
$a\in M^{\infty,1}(\mathbb{R}^{2n})$ (the Sj\"{o}strand class) if the window
$\phi$ belongs to $M^{1}(\mathbb{R}^{n})$. Other choices are certainly possible.
\end{remark}

\subsection{The connection with the KLM conditions}

In what follows we prove Theorem \ref{Theorem2}, which shows that the KLM
conditions can be recaptured by an averaging procedure from the conditions in
Theorem \ref{ThmFabio}.

\begin{proof}
[Proof of Theorem \ref{Theorem2}]Let us first observe that we can write
\[
M_{\lambda,\mu}^{(KLM)}=(2\pi\eta)^{-n}e^{-\frac{i}{2\eta}\sigma(\lambda,\mu
)}V_{\Phi}a(\tfrac{1}{2}(\lambda+\mu),J(\mu-\lambda))
\]
where $\Phi(z)=1$ for all $z\in\mathbb{R}^{2n}$. Now we have
\[
W\phi_{\nu}(z)=\left(  \tfrac{1}{\pi\eta}\right)  ^{n}e^{-\frac{1}{\eta}%
|z-\nu|^{2}}%
\]
and therefore
\[
\int_{\mathbb{R}^{2n}}W\phi_{\nu}(z)\,d\nu=\left(  \tfrac{1}{\pi\eta}\right)
^{n}\int_{\mathbb{R}^{2n}}e^{-\frac{1}{\eta}|z-\nu|^{2}}d\nu=1=\Phi
(z)\quad\forall z\in\mathbb{R}^{2n}.
\]
Hence \eqref{formula3} follows by exchanging the integral with respect to
$\nu$ in \eqref{formula3} with the integral in the definition of the STFT in
\eqref{formula2}. Fubini's theorem can be applied because the function
\[
a(z)W\phi_{\nu}(z-\zeta)
\]
belongs to $L^{1}(\mathbb{R}^{2n}\times\mathbb{R}^{2n})$ with respect to
$z,\nu$, for every fixed $\zeta\in\mathbb{R}^{2n}$.
\end{proof}

\begin{corollary}
\label{coro1} Suppose $a\in L^{1}(\mathbb{R}^{n}) \cap L^{2}(\mathbb{R}^{n})$.
With the notation in Theorem \ref{Theorem2}, suppose that $\mathcal{G}%
(\phi_{0},\Lambda)$ is a Gabor frame for $L^{2}(\mathbb{R}^{n})$. If the
matrix $(M_{\lambda,\mu}^{\phi_{0}})_{\lambda,\mu\in\Lambda,|\lambda
|,|\mu|\leq N}$ is positive semidefinite for every $N$, then so is the matrix
$(M^{(KLM)}_{\lambda,\mu})_{\lambda,\mu\in\Lambda,|\lambda|,|\mu|\leq N}$.
\end{corollary}

\begin{proof}
Observing that $\mathcal{G}(\phi_{\nu},\Lambda)$ is also a\ Gabor frame for
every $\nu\in\mathbb{R}^{2n}$, it follows from the assumptions and Theorem
\ref{ThmFabio} that the matrices $(M_{\lambda,\mu}^{\phi_{\nu}})_{\lambda
,\mu\in\Lambda,|\lambda|,|\mu|\leq N}$ are\ positive semidefinite for all
$\nu\in\mathbb{R}^{2n}$. The result therefore follows from \eqref{formula3}.
\end{proof}

\begin{corollary}
\label{coro2} Suppose $a\in L^{1}(\mathbb{R}^{n}) \cap L^{2}(\mathbb{R}^{n})$
and $\widehat{A}_{\eta}=\operatorname*{Op}_{\eta}^{\mathrm{W}}(a)\geq0$. Then,
with the notation in Theorem \ref{Theorem2}, for every finite subset
$S\subset\mathbb{R}^{2n}$ the matrix $(M^{(KLM)}_{\lambda,\mu})_{\lambda
,\mu\in S}$ is positive semidefinite (that is, the KLM conditions hold).
\end{corollary}

\begin{proof}
Since $\mathcal{G}(\phi_{0},\Lambda)$ is a frame for $L^{2}(\mathbb{R}^{n})$
for every sufficiently dense lattice $\Lambda$, as a consequence of Corollary
\ref{coro1} the matrix $(M^{(KLM)}_{\lambda,\mu})_{\lambda,\mu\in
\Lambda,|\lambda|,|\mu|\leq N}$ is positive semidefinite for all such lattices
$\Lambda$ and every integer $N$. By restricting the matrix to subspaces we see
that the submatrices $(M^{(KLM)}_{\lambda,\mu})_{\lambda,\mu\in S}$ are
positive semidefinite for every finite subset $S\subset\Lambda$. Since $a\in
L^{1}(\mathbb{R}^{n})$ the symplectic Fourier transform $a_{\sigma,\eta}$ is
continuous and therefore the same holds for every finite subset $S\subset
\mathbb{R}^{2n}$.
\end{proof}

\subsection{Almost positivity}

We now address the following question: suppose that $\mathcal{G}(\phi
,\Lambda)$ is a Gabor frame for $L^{2}(\mathbb{R}^{n})$ and assume that the
matrix $(M_{\lambda,\mu})_{\lambda,\mu\in\Lambda,|\lambda|,|\mu|\leq N}$ in
\eqref{lm1},\eqref{amunu} is positive semidefinite \textit{for a fixed $N$}.
What can we say about the positivity of the operator $\widehat{A}_{\eta}$?
Under suitable decay condition on the symbol $a$ it turns out that
$\widehat{A}_{\eta}$ is \textquotedblleft almost positive\textquotedblright%
\ in the following sense.

Let $\mathcal{G}(\phi,\Lambda)$ be a Gabor frame in $L^{2}(\mathbb{R}^{n})$,
with $\phi\in\mathcal{S}(\mathbb{R}^{n})$.

\begin{theorem}
Let $a\in M_{v_{s}}^{\infty,1}(\mathbb{R}^{2n})$ be real valued and $s\geq0$;
we use the notation $v_{s}(\zeta)=\langle\zeta\rangle^{s}$ for $\zeta
\in\mathbb{R}^{4n}$. Suppose that the matrix
\[
(M_{\lambda,\mu})_{\lambda,\mu\in\Lambda,|\lambda|,|\mu|\leq N}%
\]
in \eqref{lm1} is positive semidefinite for some integer $N$. Then there
exists a constant $C>0$ independent of $N$ such that
\[
(\widehat{A}_{\eta}\psi|\psi)_{L^{2}}\geq-CN^{-s}||\psi||_{L^{2}}^{2}%
\]
for all $\psi\in L^{2}(\mathbb{R}^{n})$.
\end{theorem}

\begin{proof}
Let $\psi\in L^{2}(\mathbb{R}^{n})$, and write its Gabor frame expansion as
\[
\psi=\sum\nolimits_{\lambda\in\Lambda,|\lambda|\leq N}c_{\lambda}%
T(\lambda)\phi+\sum\nolimits_{\lambda\in\Lambda,|\lambda|>N}c_{\lambda
}T(\lambda)\phi;
\]
denoting the sums in the right-hand side by, respectively, $\psi^{\prime}$ and
$\psi^{\prime\prime}$ we get%
\[
(\widehat{A}_{\eta}\psi|\psi)_{L^{2}}=(\widehat{A}_{\eta}\psi^{\prime}%
|\psi^{\prime})_{L^{2}}+(\widehat{A}_{\eta}\psi^{\prime}|\psi^{\prime\prime
})_{L^{2}}+(\widehat{A}_{\eta}\psi^{\prime\prime}|\psi^{\prime})_{L^{2}%
}+(\widehat{A}_{\eta}\psi^{\prime\prime}|\psi^{\prime\prime})_{L^{2}}.
\]
We have, by \eqref{aggiunta} and the positivity assumption
\[
(\widehat{A}_{\eta}\psi^{\prime}|\psi^{\prime})_{L^{2}}=\sum\nolimits_{\lambda
,\mu\in\Lambda,|\lambda|,|\mu|\leq N}c_{\lambda}\overline{c_{\mu}}%
M_{\lambda,\mu}\geq0
\]
hence it is sufficient to show that%
\begin{equation}
|(\widehat{A}_{\eta}\psi^{\prime}|\psi^{\prime\prime})_{L^{2}}|\leq
CN^{-s}||\psi||_{L^{2}}^{2}\label{5}%
\end{equation}
and similar inequalities for the other terms. Now%
\begin{equation}
|(\widehat{A}_{\eta}\psi^{\prime}|\psi^{\prime\prime})_{L^{2}}|=|(\psi
^{\prime}|\widehat{A}_{\eta}\psi^{\prime\prime})_{L^{2}}|\leq||\psi^{\prime
}||_{L^{2}}||\widehat{A}_{\eta}\psi^{\prime\prime}||_{L^{2}}.\label{6}%
\end{equation}
Observe that the function $\psi^{\prime\prime}$ has a Gabor expansion with
coefficients $c_{\lambda}=0$ for $|\lambda|\leq N$. By the frame property and
the same computation as in the proof of \eqref{aggiunta} we have
\begin{align*}
||\widehat{A}_{\eta}\psi^{\prime\prime}||_{L^{2}} &  \asymp\Vert
(\widehat{A}_{\eta}\psi^{\prime\prime}|T(\mu)\phi)_{L^{2}}\Vert_{\ell^{2}}\\
&  =\Vert{\textstyle\sum_{\lambda\in\Lambda,|\lambda|>N}}M_{\lambda,\mu
}c_{\lambda}\Vert_{\ell^{2}}.
\end{align*}
Observe that \eqref{amunu} can be rewritten in terms of the short-time
Fourier transform (STFT) on $\mathbb{R}^{2n}$ as%
\begin{equation}
a_{\lambda,\mu}=V_{W_{\eta}\phi}a(\tfrac{1}{2}(\lambda+\mu),J(\mu-\lambda)).
\label{astft}%
\end{equation}

Now, from \eqref{lm1} and \eqref{astft} we have
\[
|M_{\lambda,\mu}|=|V_{W_{\eta}\phi}a(\tfrac{1}{2}(\lambda+\mu),J(\mu
-\lambda))|.
\]
In view of the assumption $a\in M_{v_{s}}^{\infty,1}$ we have, by
\cite[Theorem 12.2.1]{Gro}, that
\[
\sum_{\nu\in\Lambda^{\prime}}\sup_{z\in\mathbb{R}^{2n}}(1+|z|+|\nu
|)^{s}|V_{W_{\eta}\phi}a(z,\nu)|<\infty
\]
for every lattice $\Lambda^{\prime}\subset\mathbb{R}^{2n}$. Now we apply this
formula with $\Lambda^{\prime}=J(\Lambda)$; using
\[
1+|\lambda|+|\mu|\asymp1+\frac{1}{2}|\lambda+\mu|+|J(\mu-\lambda)|
\]
we obtain, for $|\lambda|>N$
\begin{align*}
N^{s}|M_{\lambda,\mu}| &  \leq C\big(1+\frac{1}{2}|\lambda+\mu|+|J(\mu
-\lambda)|\big)^{s}|V_{W_{\eta}\phi}a(\tfrac{1}{2}(\lambda+\mu),J(\mu
-\lambda))|\\
&  \leq H(\mu-\lambda)
\end{align*}
for some $H\in\ell^{1}(\Lambda)$. By Schur's test we can continue the above
estimate as
\[
\Vert{\textstyle\sum_{\lambda\in\Lambda,|\lambda|>N}}M_{\lambda,\mu}%
c_{\lambda}\Vert_{\ell^{2}}\leq CN^{-s}\Vert c_{\lambda}\Vert_{\ell^{2}},
\]
which combined with (\ref{6}) gives (\ref{5}), because
\[
\Vert c_{\lambda}\Vert_{\ell^{2}}\asymp||\psi^{\prime\prime}||_{L^{2}}\leq
C^{\prime}||\psi||_{L^{2}}%
\]
and $||\psi^{\prime}||_{L^{2}}\leq C^{\prime\prime}||\psi||_{L^{2}}$.
\end{proof}


\section*{Acknowledgements}

M. de Gosson has been funded by the Grant P27773 of the Austrian Research
Foundation FWF. E. Cordero and F. Nicola were partially supported by the
Gruppo Nazionale per l'Analisi Matematica, la Probabilit\`a e le loro
Applicazioni (GNAMPA) of the Istituto Nazionale di Alta Matematica (INdAM).

\end{document}